\documentclass[a4paper,12pt,fleqn]{article}

\usepackage[nottoc, notlof, notlot]{tocbibind}
\usepackage[english]{babel}
\usepackage{fullpage}
\usepackage{amsmath,amssymb,amsfonts,amsthm}
\usepackage{graphicx}

\newtheorem{theo}{Theorem}[section]
\newtheorem{lemm}[theo]{Lemma}

\newtheorem{coro}[theo]{Corollary}
\newtheorem{prop}[theo]{Proposition}

\title{\bf Semi-infinite TASEP with a Complex Boundary Mechanism}
\author{By Nicky Sonigo}
\date{\textit{UMPA, ENS Lyon, Universit\'e de Lyon}}

\begin{document}

\newcommand{\ZZ}{\mathbb{Z}}
\newcommand{\ZZp}{\mathbb{Z}_+^*}
\newcommand{\RR}{\mathbb{R}}
\newcommand{\RRp}{\mathbb{R}_+^*}
\newcommand{\NN}{\mathbb{N}}
\newcommand{\Bcal}{\mathcal{B}}
\newcommand{\Si}[1][i]{S^{(#1)}(t)}
\newcommand{\Oi}[1][i]{\Omega^{(#1)}}
\newcommand{\Mi}[1][i]{\mu^{(#1)}_{\infty}}
\newcommand{\Ni}[1][i]{N^{(#1)}_t}
\newcommand{\Ei}[1][i]{\eta^{(#1)}_t}
\newcommand{\Ti}[1][i]{T^{(#1)}_t}
\newcommand{\TTi}[1][i]{\tilde{T}^{(#1)}_t}
\newcommand{\tend}[2][\infty]{\underset{#2 \rightarrow #1}{\longrightarrow}}
\newcommand{\Ncal}{\mathcal{N}}
\newcommand{\eqdef}{\mathrel{\mathop:}=}
\newcommand{\infset}[1]{\inf \left\{ t \geq 0 : #1\right\}}
\renewcommand{\supset}[1]{\sup \left\{ t \geq 0 : #1\right\}}
\newcommand{\indic}[1]{\mathbf{1}_{#1}}
\newcommand{\prob}[1]{\mathbf{P}\left[ #1 \right]}
\newcommand{\esp}[1]{\mathbf{E}\left[ #1 \right]}
\renewcommand{\(}{\left(}
\renewcommand{\)}{\right)}

\newlength\lgr
\setlength{\lgr}{\textwidth}
\addtolength{\lgr}{-2cm}

\maketitle


\begin{abstract}
We consider a totally asymmetric exclusion process on the positive half-line. When particles enter in the system according to a Poisson source, Liggett has computed all the limit distributions when the initial distribution has an asymptotic density. In this paper we consider systems for which particles enter at the boundary according to a complex mechanism depending on the current configuration in a finite neighborhood of the origin. For this kind of models, we prove a strong law of large numbers for the number of particles entered in the system at a given time. Our main tool is a new representation of the model as a multi-type particle system with infinitely many particle types.
\end{abstract}

\section{Introduction}

The simple exclusion process $\eta_. = (\eta_t)_{t \geq 0}$ on a countable space $S$, with random walk kernel $p(.)$, is a continuous time Markov process on $X \eqdef \{ 0,1 \}^S$. For a configuration $\eta \in X$, we say that the site $x$ is occupied (by a particle) if $\eta(x) = 1$, and is empty if $\eta(x) = 0$. A particle "tries" to move from an occupied site $x$ to an empty site $y$ at rate $p(x,y)$, or in an equivalent way, waits for an exponential time of parameter 1 and then chooses randomly a site $y$ with probability $p(x,y)$ and "tries" to jump on $y$. When the site $y$ is already occupied, the jump is canceled and the particle stays at $x$. In this way, there is always at most one particle at a given site. Formally, the exclusion process $\eta_.$ is defined as the Feller process with generator
\begin{equation} \label{gSEP}
\Omega f(\eta) \eqdef \sum_{x,y \in S} p(x,y) \eta(x) \left( 1-\eta(y) \right) \left[ f(\eta_{x,y}) - f(\eta) \right],
\end{equation}
for all cylinder functions $f$, where
\begin{equation*}
\eta_{x,y}(z) \eqdef \left\{ \begin{array}{ll}
					\eta(y) & \text{if } z=x,\\
					\eta(x) & \text{if } z=y,\\
					\eta(z) & \text{otherwise.}
				       \end{array} \right.
\end{equation*}
A natural question is to describe the set of invariant probability measures $\mathcal{I}$ which is the set of probability measures $\mu$ on $S$ such that, if $\eta_0 \sim \mu$ then for all $t \geq 0$, $\eta_t \sim \mu$. These measures are characterized  by the equations:
\begin{equation*}
\int \Omega f \mu(d\eta) = 0,
\end{equation*}
for any cylinder functions $f$ (see e.g. \cite{Liggett1985} for a review). In the case $S = \ZZ$, the set of translation invariant stationary measures is exactly the set of convex combinations of product Bernoulli measures on $\ZZ$ (see \cite{Liggett1976}).

\bigskip

In this paper, we consider the case $S \eqdef \ZZp$, or $S \eqdef \left\{ 1,\ldots,R \right\}$ with $R \geq 2$, and $p(x,x+1) \eqdef 1$, i.e., the totally asymmetric nearest neighbor case. In $\ZZp$, one has to add some boundary mechanism to make the model non trivial. The simplest way to do this is to add a particle reservoir at site $0$ with a certain density $\lambda > 0$. This means that a new particle is created at site $1$ according to a Poisson process with rate $\lambda$ when this site is empty. We call the model on $\ZZp$ the TASEP($\lambda$), and we denote by $\Omega_{\lambda}$ its generator and by $S_{\lambda}(t)$ its semi-group:
\begin{equation} \label{gsTASEP}
\Omega_\lambda f(\eta) = \lambda \left( 1-\eta(1) \right) \left[ f(\eta_1) - f(\eta) \right] + \sum^{\infty}_{x = 1} \eta(x) \left( 1-\eta(x+1) \right) \left[ f(\eta_{x,x+1}) - f(\eta) \right],
\end{equation}
for all cylinder functions $f$, where
\begin{equation*}
\eta_1(z) \eqdef \left\{ \begin{array}{ll}
					1-\eta(1) & \text{if } z=1,\\
					\eta(z) & \text{otherwise.}
				       \end{array} \right.
\end{equation*}
In \eqref{gsTASEP} we see two parts for the generator: one is due to the boundary mechanism and we will call it the \emph{boundary part}; the other one, which has the form given by \eqref{gSEP} for $S=\ZZp$, is due to the exclusion process and we will call it the \emph{bulk part}.
In the same way, if the state space is $\left\{ 1,\ldots,R \right\}$, then we add a particle reservoir with density $\rho \in \left[ 0,1 \right]$ at site $R+1$, which means that when the site $R$ is occupied, then the particle in this site disappears with rate $1-\rho$.

Let us introduce some notation. In the following, we denote by $\nu^\lambda$ the product measure on $\ZZp$ with density $\lambda$ and by $\tau$ the shift. $\tau$ acts on configurations $\eta \in X$ by
\begin{equation*}
\tau \eta(x) \eqdef \eta(x+1), \forall x \in \ZZp,
\end{equation*}
on functions $f : X \longrightarrow \RR$ by
\begin{equation*}
\tau f(\eta) \eqdef f(\tau \eta), \forall \eta \in X,
\end{equation*}
and on measures $\mu$ on X by
\begin{equation*}
\int f d\tau\mu \eqdef \int \tau f d\mu, \forall f \in L^1(\mu).
\end{equation*}
For a measure $\mu$ on $S$ and $f \in L^1(\mu)$, we will denote $\left\langle f \right\rangle_\mu = \int f d\mu$.\\

\bigskip

We are interested in the limit behavior of the distribution at time $t$. For this model, we have a good understanding about what happens at equilibrium. Indeed, Liggett has shown in \cite{Liggett1975} the following ergodic theorem, which gives the limit measure for an initial measure with a product form and an asymptotic density:
\begin{theo}[Liggett  \cite{Liggett1975}] \label{th1}
Let $\pi$ be a product measure on $\ZZp$ for which\\
$\rho \eqdef \lim_{x \to \infty} \left\langle \eta(x) \right\rangle_{\pi}$ exists.
\begin{align*}
&\text{If } \lambda \geq \frac{1}{2} \text{ then } \lim_{t \to \infty} \pi S_{\lambda}(t) =
\begin{cases}
  \mu^{\lambda}_{\rho}, \text{ if } \rho \geq \frac{1}{2} \text{ (bulk dominated),}\\
  \mu^{\lambda}_{\frac{1}{2}}, \text{ if } \rho \leq \frac{1}{2} \text{ (maximum current).}
\end{cases}\\
&\text{If } \lambda \leq \frac{1}{2} \text{ then } \lim_{t \to \infty} \pi S_{\lambda}(t) =
\begin{cases}
  \mu^{\lambda}_{\rho}, \text{ if } \rho > 1-\lambda \text{ (bulk dominated),}\\
  \nu^{\lambda}, \text{ if } \rho \leq 1-\lambda \text{ (boundary dominated),}
\end{cases}
\end{align*}
where the $\mu^{\lambda}_{\rho}$'s, for $\rho \geq \frac{1}{2}$, are stationary measures and asymptotically product with density $\rho$, i.e., $\lim_{x \to \infty} \tau^x \mu^{\lambda}_{\rho} = \nu^{\rho}$ (in a weak sense with test functions $f \in C(X,\mathbb{R})$). We also have $\mu^{\lambda}_{\lambda} = \nu^{\lambda}$.
\end{theo}

\bigskip

To describe the set of invariant probability measures in the cases $S = \ZZ$ and $S = \ZZp$, he uses extensively that the product Bernoulli measures are invariant and for these mesures one can make explicit computations. To break the product form of the invariant measures, one can consider a boundary mechanism which is not Poisson. We limit ourselves to finite range boundary mechanism, i.e., for which there exist some $R \in \ZZp$ such that the boundary part of the generator vanishes on every cylinder function with support in $\{R+1,\ldots\}$. This idea was first introduced by Grosskinsky in chapter 3 of \cite{Grosskinsky2003} where he defined the following Feller process:
\begin{align}
\begin{split} \label{gscTASEP}
\Omega f(\eta) &\eqdef \sum_{x \in \ZZ^*_+} \eta(x) \left( 1-\eta(x+1) \right) \left[ f(\eta_{x,x+1}) - f(\eta) \right]\\
		        &\hspace{4mm} + \sum_{\xi, \xi' \in X_R} d_{\xi,\xi'} \indic{\eta_{|S_R}=\xi} \left[ f(\xi'\cup\eta_{|^cS_R}) - f(\eta) \right],
\end{split}
\end{align}
for all cylinder functions $f$ where $S_R \eqdef \left\{ 1,\ldots,R \right\}$, $X_R \eqdef \left\{ 0,1 \right\}^{S_R}$ and $\xi' \cup \eta_{|^cS_R}$ is the natural concatenation of configurations on $S_R$ and on $^cS_R$. In \eqref{gscTASEP}, the first sum describes the bulk part and the second sum the boundary part of the dynamic.

\bigskip

The reason for which we only treat the finite range case is that when we are not in this case, pathological things can occur. For example, consider the following dynamic with non local boundary mechanism: Define the asymptotic density of a configuration $\eta \in X$ by $\rho(\eta) \eqdef \liminf_{x\to\infty} \frac{1}{x} \sum_{i=1}^x \eta(i)$; we consider now a TASEP on $\ZZp$ for which the rate of apparition of a particle in site $1$ is $\rho(\eta)$ where $\eta$ is the current configuration. More formally, the boundary part of the generator is
\begin{equation*}
\rho(\eta)(1-\eta(1))\left[ f(\eta_1) - f(\eta) \right].
\end{equation*}
In this example, every mixture of product Bernoulli measures is invariant for the process.

For this generalized boundary mechanism, we will not have an exact solution as for the TASEP($\lambda$). Our approach is to study the number of particles entered in the system up to time $t$. We will see that it grows linearly with an almost sure speed equal to the stationary current $j_\infty$. Define $\rho_\infty$ as the root of $\rho(1-\rho) = j_\infty$ in $[0,1/2[$. We think that the process has a stationary measure which is asymptotically product with density $\rho_\infty$ but we are still unable to prove it.

\bigskip

The rest of the paper is organized as follow: in section \ref{GraphicalConstruction} we give a construction of the process defined above using a graphical representation similar to that introduced by Harris \cite{Harris1978}. We also introduce the basic coupling technique which is the main tool used in the paper; in section \ref{AttractiveCase} we give some general results on the asymptotic behavior of the TASEP with complex boundary mechanism; finally, in section \ref{ParticularCase} we study a particular example: take a TASEP($\lambda$) on $\ZZp$ and add a source (independent of everything) with density $\epsilon > 0$ which is activated only when site $2$ is occupied. For this model, let $N_t$ be the number of particles in the system at time $t$ when we start from the empty configuration. Then the main result of this paper is the following strong law of large numbers:
\begin{theo} \label{MainResult}
Almost surely,
\begin{equation*}
\lim_{t \rightarrow \infty} \frac{N_t}{t} = \lambda(1-\lambda) + \lambda(1-\lambda)p(\lambda)\epsilon + o(\epsilon),
\end{equation*}
where $p(\lambda)$ is a positive constant (depending only on $\lambda$) for which we give a natural probabilistic interpretation.
\end{theo}

It should be noted that this particular choice of boundary mechanism is rather arbitrary, and that our method is robust enough to be used in a much larger generality. However, the notations which would be needed would be much more tedious, while providing very little additional insight into the model --- so we choose to limit ourselves to one representative case.

\section*{Acknowledgments}
I wish to thank Vincent Beffara for his help and suggesting this work, and Stefan Gro\ss kinsky for useful discussions.


\section{The Harris construction} \label{GraphicalConstruction}

We will use the method developed by Harris  \cite{Harris1978} to construct our process. Let
\[
\Ncal \eqdef \left( \Ncal_x, \Ncal_{\eta,\eta'}, x \in \ZZp, \eta, \eta' \in \left\{0,1\right\}^{\left\{1,\ldots, R\right\}} \right),
\]
be a family of independent Poisson point processes on $\RRp$ constructed on the same probability space $\left( \Gamma,\mathcal{F},\mathbf{P} \right)$, such that the rate of processes indexed by $\ZZ_+$ is 1 and the rate of the process indexed by $\left( \eta,\eta' \right)$ is $d_{\eta,\eta'} \geq 0$. By discarding a $\mathbf{P}$--null set, we may assume that
\begin{equation} \label{hyp}
\parbox{\lgr}{each poisson point process in $\Ncal$ has only finitely many jump times in every bounded interval $\left[ 0,T \right]$, and no two distinct processes have a jump in common.}
\end{equation}
We denote
\[
\Ncal_0 \eqdef \bigcup_{\eta, \eta' \in \left\{ 0,1 \right\}^{\left\{ 1,\ldots,R \right\}}} \Ncal_{\eta,\eta'}
\]

Fix $T > 0$ and $\eta \in X$. The process $\left( \eta_t \right)_{0 \leq t \leq T}$ starting from $\eta$ is now constructed as follows. Consider the following subgraph of $\ZZ_+$:
\[
\mathcal{G}_T \eqdef \Bigl\{ \left\{ x,x+1 \right\} : x \geq R, \Ncal_x \cap \left[ 0,T \right] \neq \emptyset \Bigr\} \bigcup \Bigl\{ \left\{ x,x+1 \right\} : x \in \left\{ 0,\ldots,R-1 \right\} \Bigr\}.
\]
It is easy to see that every connected component of $\mathcal{G}_T$ is almost surely finite. Let $\Gamma_0$ be the subset of $\Gamma$ such that \eqref{hyp} and the above condition hold for all $T \geq 0$. Then we have $\Gamma_0 \in \mathcal{F}$ and $\prob{\Gamma_0} = 1$. We consider now only $\omega \in \Gamma_0$. For every connected component $\mathcal{C}$ of $\mathcal{G}_T$, the set $\left( \cup_{x \in \mathcal{C}} \Ncal_x\right) \bigcap \left[ 0,T \right]$ is finite so its elements can be ordered chronologically $\tau_1 < \ldots < \tau_n$ and we need only to describe the action of each of them. We start with the configuration $\eta$:
\[
\eta_t(x) := \eta(x)
\] for all $x \in \mathcal{C}$ and $0 \leq t < \tau_1$.\\
Suppose that the process is constructed on $\mathcal{C}$ for $0 \leq t < \tau_k$ and $k \in \left\{1,\ldots, n\right\}$. Then:
\begin{itemize}

\item if $\tau_k \in \Ncal_{\xi,\xi'}$ and if $\eta_{\tau_k^- |S_R} = \xi$ then $\eta_{\tau_k |S_R} \eqdef \xi'$ and $\eta_{\tau_k}(x) \eqdef \eta_{\tau_k^-}(x)$ for all $x \in \mathcal{C}\backslash S_R$,\\

\item if $\tau_k \in \Ncal_{\xi,\xi'}$ and if $\eta_{\tau_k^- |S_R} \neq \xi$ then $\eta_{\tau_k}(x) \eqdef \eta_{\tau_k^-}(x)$ for all $x \in \mathcal{C}$,\\

\item if $\tau_k \in \Ncal_x$ and $\eta_{\tau_k^-}(x) \( 1-\eta_{\tau_k^-}(x+1) \) = 1$ then $\eta_{\tau_k} \eqdef \( \eta_{\tau_k^-} \)_{x,x+1}$ on $\mathcal{C}$,\\

\item  if $\tau_k \in \Ncal_x$ and $\eta_{\tau_k^-}(x) \( 1-\eta_{\tau_k^-}(x+1) \) \neq 1$ then $\eta_{\tau_k} \eqdef \eta_{\tau_k^-}$ on $\mathcal{C}$,\\

\end{itemize}
Finally, we put $\eta_t \eqdef \eta_{\tau_k}$ on $\mathcal{C}$ for $\tau_k \leq t < \tau_{k+1}$ if $k < n$ and for $\tau_n \leq t \leq T$ if $k=n$. We make the same construction on every connected component of $\mathcal{G}_T$ and then let $T$ go to infinity to get the process $\( \eta_t \)_{t \geq 0}$ for every $\omega \in \Gamma_0$.

The usefulness of such a construction is that, using the same Harris process, we can construct two or more realizations of the process on the same probability space starting from different initial configurations. We will refer to this coupling as the \emph{basic coupling}.


\section{The attractive case} \label{AttractiveCase}

In this section, we consider the process with the generator:
\begin{align} \label{gscTASEP2}
\begin{split}
\Omega f(\eta) &\eqdef \sum_{x \in \ZZ^*_+} \eta(x) \left( 1-\eta(x+1) \right) \left[ f(\eta_{x,x+1}) - f(\eta) \right]\\
		        &\hspace{4mm} + \sum_{\xi, \xi' \in X_R} d_{\xi,\xi'} \indic{\eta_{|S_R}=\xi} \left[ f(\xi'\cup\eta_{|^cS_R}) - f(\eta) \right],
\end{split}
\end{align}
and we assume the process is attractive.


\subsection{The stationary measure}

\begin{prop} \label{pr2}
Suppose that the process is monotone (or attractive). Then, starting from the empty configuration, the measure at time $t$ of the process, say $\mu_t$, is stochastically increasing and converges to a measure $\mu_{\infty} \in \mathcal{I}$ which is the smallest invariant measure of the process. Furthermore, $\mu_{\infty} \in \mathcal{I}_e$ and $\mu_{\infty}$ is ergodic.
\end{prop}

\begin{proof}
Let $0 \leq s < t$. We have $\delta_0 \prec \mu_{t-s}$. Thus by attractivity of the process, we have  $\delta_0S(s) \prec \mu_{t-s}S(s)$, i.e.,  $\mu_s \prec \mu_t$. Hence, by monotonicity, $\mu_t$ converges weakly to an invariant measure $\mu_{\infty}$.

For all $\nu \in \mathcal{I}$, we have $\delta_0 \prec \nu$, which implies that $\mu_t \prec \nu$ for all $t \geq 0$, and then $\mu_{\infty} \prec \nu$. Assume now that $\mu_{\infty} = \lambda\nu_1 + \( 1-\lambda \)\nu_2$, with $\nu_1, \nu_2 \in \mathcal{I}$ and $\lambda \in \left] 0,1 \right[$. We have $\mu_{\infty} = \lambda\nu_1 + \( 1-\lambda \)\nu_2 \succ \mu_{\infty}$, thus $\nu_1 = \nu_2 = \mu_{\infty}$ and $\mu_{\infty}$ is extremal. Finally, by Theorem B52 of \cite{Liggett1999}, $\mu_{\infty}$ is also ergodic.
\end{proof}

\begin{prop} \label{exist}
$\tau^R \mu_\infty$ is stochastically dominated by the product Bernoulli measure with density $\frac{1}{2}$.\end{prop}

\begin{proof}
Define $\Ncal' \eqdef \( \Ncal'_x, x \in \ZZ_+ \)$, where $\Ncal'_x \eqdef \Ncal_{x+R}$. Then $\Ncal'$ defines a TASEP $\( \xi_t \)$ on $\ZZp$ with rate $1$ of particle apparition in $1$. By theorem \ref{th1}, starting from the empty configuration, the distribution at time $t$ converges to $\nu^{\frac{1}{2}}$. In this coupling, we have $\xi_t(x) \geq \eta_t(x+R)$ almost surely for all $t \geq 0$ and $x \geq 1$. Thus the restriction of $\mu_{\infty}$ to $\left\{ R+1, R+2,\ldots \right\}$ is stochasticaly dominated by $\nu^{\frac{1}{2}}$.
\end{proof}


\subsection{Asymptotic measures}

Let us extend the measure $\mu_\infty$ on a measure on $\{0,1\}^\ZZ$ by
\begin{equation*}
\overline{\mu}_\infty (A) \eqdef \mu_\infty \left\{ \eta \in X : \tilde{\eta} \in A \right\},
\end{equation*}
where $\tilde{\eta}(x) \eqdef \begin{cases}
						\eta(x) \text{ if } x \geq 1,\\
						0 \text{ otherwise},
				   	    \end{cases}$\\
for all A in the cylindric field of $\{0,1\}^\ZZ$. We will still denote this measure $\mu_\infty$. Let $\tilde{\mu}_k \eqdef \tau^k \mu_\infty$ and consider any weak limit $\tilde{\mu}_\infty$ of this sequence:
\begin{equation*}
\lim_{i \to \infty} \tilde{\mu}_{k_i} = \tilde{\mu}_\infty.
\end{equation*}

\begin{prop} \label{invtasep}
The measure $\tilde{\mu}_\infty$ is a translation invariant stationary measure for the TASEP on $\ZZ$. Consequently, it is a mixture of product Bernoulli measures, i.e., there exists a probability measure $\sigma$ on $\left[ 0,1 \right]$ such that
\begin{equation*}
\tilde{\mu}_\infty = \int_0^1 \nu^\lambda \sigma(d\lambda).
\end{equation*}
\end{prop}

\begin{proof}
Let $\Omega^e$ be the generator of the TASEP on $\ZZ$. For any cylinder function $f : \left\{ 0,1 \right\}^\ZZ \to \RR$, let $x \in \ZZp$ large enough such that $supp \hspace{1mm} \tau^x f \subset \left\{ R+1,R+2,\ldots \right\}$, where $supp f$ is the support of $f$. Thus $\tau^x f$ could be considerate has a function on $\ZZp$ and we can apply the generator $\Omega$ to this function. We get $\Omega \tau^y f = \Omega^e \tau^y f$ for all $y \geq x$. But it is easy to see that $\Omega^e$ and $\tau$ commute, thus we have
\begin{align*}
\int \Omega \tau^y f \mu_\infty(d\eta) = 0 &= \int \tau^y \Omega^e f \mu_\infty(d\eta),\\
							       &= \int \Omega^e f \tilde{\mu}_y(d\eta).
\end{align*}
Hence for $i$ large enough, $\left\langle \Omega^e f \right\rangle_{\tilde{\mu}_i} = 0$, which implies that $\left\langle \Omega^e f \right\rangle_{\tilde{\mu}_\infty} = 0$. This is true for arbitrary $f$ thus $\tilde{\mu}_\infty$ is invariant for the TASEP on $\ZZ$. We know that for this model we have $\mathcal{I}_e = \left\{ \nu^\lambda, \lambda \in \left[ 0,1 \right] \right\} \cup \left\{ \nu_n, n \in \ZZ \right\}$, where $\nu_n = \tau^n\nu_0$ and $\nu_0$ is the Dirac measure of the configuration for which all the sites $x \geq 0$ are occupied and all the sites $x < 0$ are empty (see \cite{Liggett1976}). Using Proposition \ref{exist}, $\tilde{\mu}_\infty$ is stochastically dominated by $\nu^{\frac{1}{2}}$, thus $\tilde{\mu}_\infty$ is translation invariant and is a mixture of product Bernoulli measures.
\end{proof}


\subsection{A strong law of large numbers}

Let $\mu$ be an invariant and ergodic measure for the process with the generator given by \eqref{gscTASEP2}. Fix $\xi_0, \xi$ and $\xi'$ three configurations on $S_R$ and consider
\[
N(t) \eqdef \sharp \( \Ncal_{\xi, \xi'} \cap I_t \),
\]
where $I_t \eqdef \left\{ s \in \left[ 0,t \right] : \eta_{s|S_R} = \xi_0 \right\}$. We will show a strong law of large numbers for $N(t)$ which will be useful in the sequel.

\begin{prop} \label{LLG}
If the process starts from $\mu$ and if $\xi' \neq \xi_0$, then almost surely:
\begin{equation*}
\lim_{t \to \infty} \frac{N_t}{t} = d_{\xi, \xi'} \mu \left\{ \eta \in X : \eta_{s|S_R} = \xi_0 \right\}.
\end{equation*}
\end{prop}

\begin{proof}
Let
\[
T_t \eqdef \int_0^t \indic{\eta_{s|S_R} = \xi_0} ds,
\]
and
\[
\psi(t) \eqdef \inf \left\{ s \geq 0 : T_s = t \right\}.
\]
Since $\mu$ is ergodic, $\frac{T_t}{t} \tend{t} \mu \left\{ \eta \in X : \eta_{s|S_R} = \xi_0 \right\}$ almost surely. Let $I \eqdef \{t \geq 0 : \eta_{t|S_R} = \xi_0\}$. $\psi : \RR_+ \to I$ is a one to one map so we can define $\mathcal{M} \eqdef \psi^{-1}(\Ncal_{\xi,\xi'} \cap I)$ and $N'(t)$ the associated counting process. We have now that $N'(t) = N(\psi(t))$ almost surely and $\mathcal{M}$ is a Poisson point process with parameter $d_{\xi,\xi'}$. Indeed, for all $t \geq 0$, the next point in $\mathcal{M}$ after $t$ is equal to $t$ plus an exponential variable with parameter $d_{\xi,\xi'}$, and after this point, the process is independent of the past (see Figure \ref{fig:LGN}). Then, it is easy to see that for every $\epsilon > 0$, $\psi(T_t) \leq t \leq \psi(T_t + \epsilon)$. Since $N(t)$ is non-decreasing, we get $N'(T_t) \leq N(t) \leq N'(T_t + \epsilon)$. Consequently:
\[
\frac{N'(T_t)}{T_t} \frac{T_t}{t} \leq \frac{N(t)}{t} \leq \frac{N'(T_t + \epsilon)}{T_t + \epsilon} \frac{T_t + \epsilon}{t}.
\]
Since both sides converge to $d_{\xi, \xi'} \mu \left\{ \eta \in X : \eta_{s|S_R} = \xi_0 \right\}$ almost surely, it leads to the conclusion.
\end{proof}

\begin{figure}[h]
   \centering
   \includegraphics[width=0.7\textwidth]{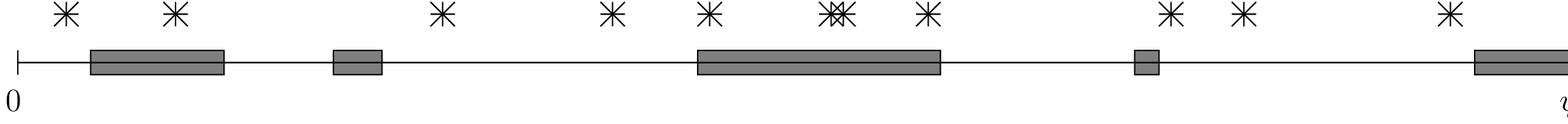} 
   \parbox[c]{12cm}{\caption{\textit{On the time interval $[0,\psi(t)]$ we see the set $I_{\psi(t)}$ in grey. The total length of the grey part is $t$. The stars are points of the process $\Ncal_{\xi,\xi'}$. In this example, $N'(t) = 5$.}}
   \label{fig:LGN}}
\end{figure}

\section{A particular case and the Multi-Species model} \label{ParticularCase}
 
In this section, we are interested in a particular case of TASEP with a complex boundary mechanism: let $\lambda, \epsilon > 0$ such that $\lambda+\epsilon < \frac{1}{2}$. Particles are created at site $1$ with rate $\lambda +\epsilon\eta(2)$, where $\eta$ is the current configuration and the bulk dynamic is the one of the TASEP. This model has a generator given by:
\begin{align}
\begin{split} \label{tm}
\Omega f(\eta) &= \sum_{x \in \ZZp} \eta(x) \( 1-\eta(x+1) \) \left[ f(\eta_{x,x+1}) - f(\eta) \right]\\
		        &\hspace{5mm} \times \( 1-\eta(1) \) \( \lambda+\epsilon\eta(2)\) \left[ f(\eta_1) - f(\eta) \right],
\end{split}
\end{align}
for all cylinder functions $f$ on $X$. As it is explained in the introduction, the choice of the model is rather arbitrary, and the methods that we use are quite robust (at least as long as the system can be dominated by a product Bernoulli measure of intensity lower than $1/2$ --- which is indeed the case here).

\begin{figure}[h]
   \centering
   \includegraphics[width=0.6\textwidth]{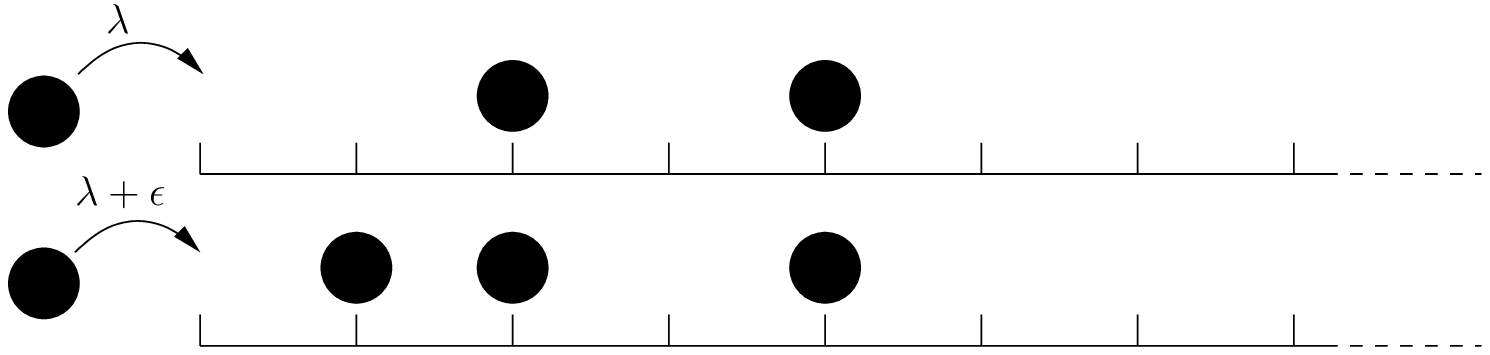} 
   \parbox[c]{12cm}{\caption{\textit{Particles enter with additional rate $\epsilon$ when the site $2$ is occupied.}}
   \label{fig:TasepMod}}
\end{figure}

In this model, the range of the boundary mechanism is $R = 2$. The hypothesis $\epsilon > 0$ implies that the process is monotone, thus we can define the smallest stationary measure $\mu_\infty=\mu_\infty \( \lambda,\epsilon \)$ of the model. Using the Harris representation, we can couple this process with $\eta^\lambda_.$, a TASEP($\lambda$), and $\eta^{\lambda+\epsilon}_.$, a TASEP($\lambda+\epsilon$), such that if $\eta^{\lambda}_0 \leq \eta_0 \leq \eta^{\lambda+\epsilon}_0$ then for all $t \geq 0$, $\eta^{\lambda}_t \leq \eta_t \leq \eta^{\lambda+\epsilon}_t$. This proves that $\nu^\lambda \prec \mu_\infty \prec \nu^{\lambda+\epsilon}$ and then $\nu^\lambda \prec \tilde{\mu}_\infty \prec \nu^{\lambda+\epsilon}$.


\subsection{Some estimates about the flux}

Here we will see another way to see the process with the generator given by \eqref{tm}. For any $i \geq 1$, let
\begin{equation*}
\mathcal{X}_i \eqdef \left\{ \infty,1,\ldots,i \right\}^{\ZZp}.
\end{equation*}
We define
\begin{align}
\begin{split} \label{gms}
\Oi f (\eta) \eqdef &\lambda \indic{\eta(1) \geq 2} \left[ f \(\eta_{1 \to 1}\) - f(\eta) \right]\\
			   &+ \sum_{j=2}^i \epsilon \indic{\eta(1) \geq j+1} \indic{\eta(2) = j-1} \left[ f \(\eta_{j \to 1}\) - f(\eta) \right]\\
			   &+ \epsilon \indic{\eta(1) = \infty} \indic{\eta(2) = i} \left[ f\(\eta_{i \to 1}\) - f(\eta) \right]\\
			   &+ \sum_{x=1}^\infty \indic{\eta(x) \neq \infty} \indic{\eta(x+1) \geq \eta(x)+1} \left[ f\(\eta_{x,x+1}\) - f(\eta) \right],
\end{split}
\end{align}
for all cylinder function $f : \mathcal{X}_i \to \RR$, where
\begin{equation*}
\eta_{j \to 1}(x) \eqdef \begin{cases}
					 j \text{ if } x=1,\\
					 \eta(x) \text{ otherwise,}
		     		  \end{cases}
\end{equation*}
for $j \in \{1,\ldots,i\}$.

\begin{figure}[h]
\centering
\includegraphics[width=0.6\textwidth]{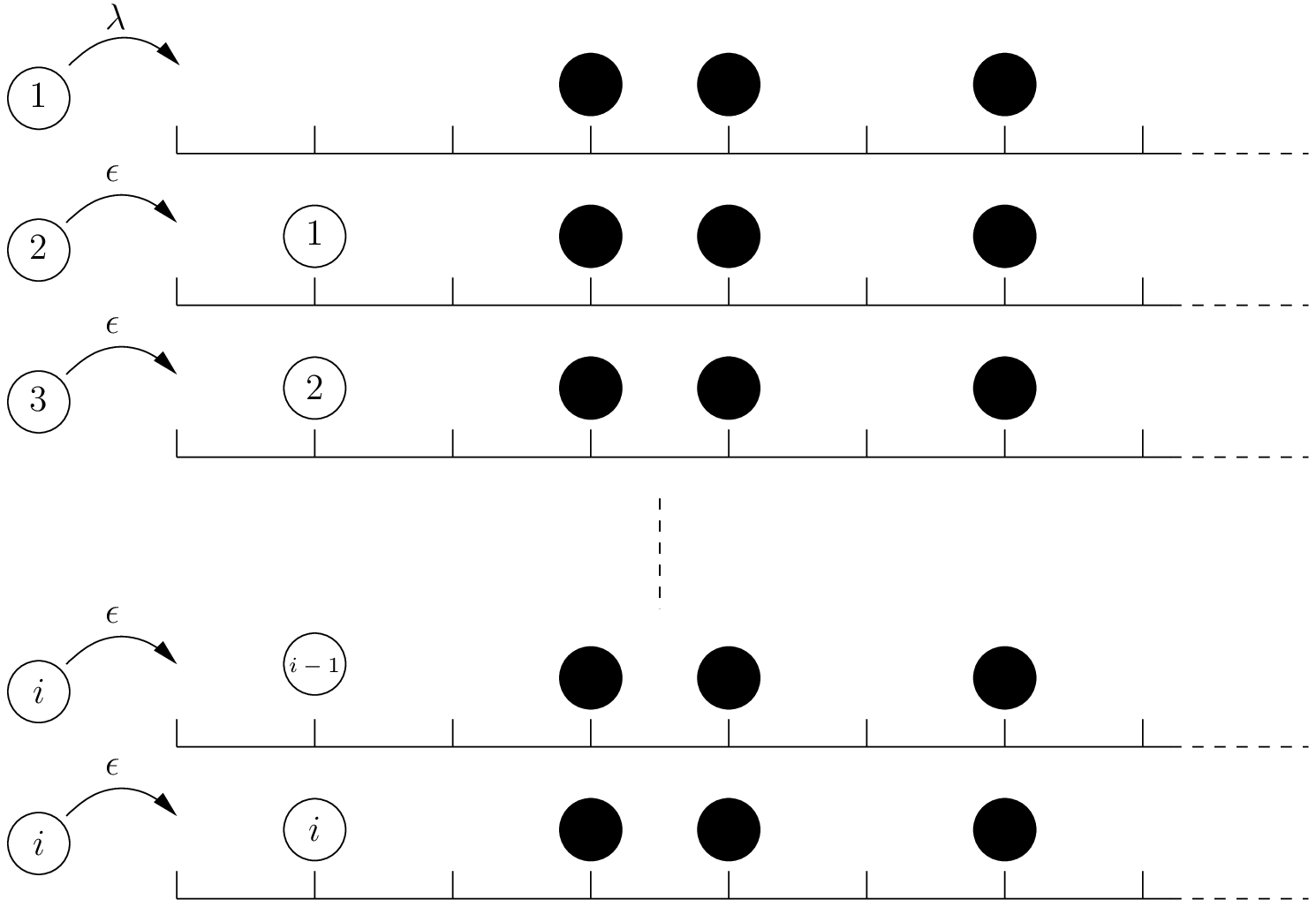}
\parbox[c]{12cm}{\caption{\textit{First class particles enter with rate $\lambda$ whatever is the configuration in $\left\{ 2,3,\ldots \right\}$ and second class particles enter with rate $\epsilon$ if the site $2$ is occupied by a first class particle. Particles in black are indistinguishable particles (their class has no influence on the rate of the source in the current configuration).}}
\label{fig:2ClassesModel}}
\end{figure}

We fix $i \geq 2$ for the sequel. The new description is described in Figure \ref{fig:2ClassesModel} and in the following. We put the particles into a certain number of classes. For a configuration $\eta \in \mathcal{X}_i$ and for a site $x \in \ZZp$, the number $\eta(x)$ designates the class of the particle at site $x$ if it exists, i.e., if $\eta(x) \neq \infty$, and is equal to $\infty$ if the site is empty. We use here another notation for empty sites because it allows us to have a simpler expression for the generator and we can also interpret holes as particles of class infinity. The evolution is the same as before, except that if a particle of the $k$-th class (or of type $k$) attempts to jump on a site occupied by a particle of the $j$-th class (or of type $j$), then it is not allowed to do so if $k \geq j$, and the particles exchange positions if $k < j$. We say that a particle of class $k \in \left\{1, 2, \ldots \right\}$ has priority on all particles of classes greater than $k$. In this way, a particle of type $k$ behaves as a hole for particles of type $j < k$.

Now we will explain how we affect classes to the particles. First class particles enter in the system (in the site $1$) at rate $\lambda$. As they have priority on other particles, they are not affected by them, so the process of first class particles is simply a TASEP($\lambda$) on $\ZZp$. Next, particles of class $2 \leq j \leq i-1$ enter in the system with rate $\epsilon$, if the site $2$ is occupied by a particle of class $j-1$ and with rate $0$ otherwise. Finally, particles of class $i$ enter in the system with rate $\epsilon$ if the site $2$ is occupied by a particle of class $i-1$ or $i$ and with rate $0$ otherwise. For each configuration of the system, at most $2$ types of particles are allowed to enter in the system. We can also remark that if we consider the process consisting with particles of class $1,\ldots,i$, then it has the generator given by \eqref{tm}.

\bigskip

In terms of Harris system, we define $\Ncal$ the set of the following independent Poisson point processes on $\RRp$: let $ \(\Ncal_x, x \geq 1\)$ be Poisson point processes of rate 1; let $\(\Ncal^b_j, j \geq 1\)$ be Poisson point processes of rate $\lambda$ for $\Ncal^b_1$ and of rate $\epsilon$ for the others. The mechanism is then the following: if $t \geq 0$ is a jumping time of $\Ncal_x$ and if at time $t^-$ we have $\eta(x) \neq \infty$ and $\eta(x+1) > \eta(x)$ (i.e., there is a particle at $x$ and it has higher priority than the one at $x+1$ if it exists), then the particle at $x$ jumps to $x+1$ (and the particle at $x+1$ jumps to $x$ if it exists); if $t \geq 0$ is a jumping time of $\Ncal^b_1$ and if at time $t^-$ we have $\eta(1) \geq 2$, then a first class particle appears at site $1$; if $t \geq 0$ is a jumping time of $\Ncal^b_j$ with $2 \leq j \leq i-1$ and if at time $t^-$ we have $\eta(1) \geq j+1$ and $\eta(2) = j-1$, then a $j$-particle appears at site $1$; finally, if $t \geq 0$ is a jumping time of $\Ncal^b_i$ and if at time $t^-$ we have $\eta(1) = \infty$ and $\eta(2) \in \left\{ i-1, i \right\}$, then a $i$-particle appears at site $1$.

We denote by $S^{(i)}(t)$ the semi-group corresponding to the generator $\Oi$ and by $(\Ei[j])_{t \geq 0}$ the process of the $j$-th class particles for $j=1,\ldots,i$. It is easy to see that the distribution of the process $(\Ei[j])_{t \geq 0}$ for $j \in \left\{ 1,\ldots,i-1 \right\}$ does not depend on $i$. The process is attractive, thus we can define $\mu^{(i)}_{\infty}$ as the weak limit of $\delta_\infty \Si$. As in Proposition \ref{pr2}, this measure is extremal, ergodic and the smallest invariant measure of the system. For all $1 \leq j \leq i$, we note $\bar{\eta}^{(j)}_t \eqdef \sum^j_{k=1}\Ei[k]$. Remark that the process $(\bar{\eta}^{(i)}_t)_{t \geq 0}$ is exactly the process that we want to study, i.e., it has the generator given by \eqref{tm}. Furthermore, for all $j \in \left\{ 1,\ldots,i \right\}$ the distribution of the process $(\bar{\eta}^{(j)}_t)_{t \geq 0}$ is the same for all $i$.

\bigskip

In order to compare the processes $(\bar{\eta}^{(i-1)}_t)_{t \geq 0}$ and $(\bar{\eta}^{(i)}_t)_{t \geq 0}$, we need to control the number of particles of a given type in the system at a given time. Let $\Ni[j]$ be the number of $j$-particles entered in the system between times $0$ and $t$, and define
\begin{equation*}
\Ti[j] \eqdef \int_0^t \eta^{(j)}_s(2)ds
\end{equation*}
and
\begin{equation*}
\tilde{T}_t \eqdef \int_0^t \eta_s^{(1)}(2) \( 1-\eta_s^{(1)}(1) \) ds,
\end{equation*}
for all $j \in \left\{ 1,\ldots,i \right\}$.

$\Ti[j]$ is the time spent by $j$-particles in site $2$ during $\left[ 0,t \right]$, and $\tilde{T}_t$ is the length of the subset of $\left[ 0,t \right]$ for which $2$-particles can enter in $1$ with rate $\epsilon$ (excepted if the site $1$ is already occupied by another $2$-particle). The following lemma says that we have an uniform control on the total time spent by a particular particle of type $\geq 2$ at site $2$. Let $T_{k,j}$ be the total time spent in site $2$ by the $k$-th particle of type $j \geq 2$ entered in the system.

\begin{lemm} \label{lem1}
There exists a constant $C_{\lambda} \in \left] 0,+\infty \right[$ such that for all $k \geq 1$ and all $j \geq 2$ we have
\begin{equation*}
\esp{T_{k,j}} \leq C_{\lambda}.
\end{equation*}
\end{lemm}

\begin{proof}
Let $E_t$ be the event that, between times $t$ and $t+1$, a first class particle enters (or tries to enter) in the system, then it jumps, if it is possible, to the site $2$ and finally another first class particle enters in the system. We also assume that in $E_t$ there are no other jumping times for $\Ncal_1$, $\Ncal_2$ and $\Ncal^b_1$ between $0$ and $t$. In particular, if $E_t$ occurs and if there was a particle of type $\geq 2$ in site $2$ at time $t$, then it has disappeared at time $t+1$. We have that $\prob{E_t} = q(\lambda)$ does not depend on $t$ and $q(\lambda) > 0$.

On the event $\left\{ T_{k,j} > t \right\}$, there exists a time $\tau$ such that the $k$-th particle of type $j$ is at the site $2$ and it has spent exactly a time $t$ in this site between $0$ and $\tau$. We have $E_{\tau} \subset \left\{ T_{k,j} \leq t+1 \right\}$. Hence
\begin{equation} \label{eq4}
\prob{E_{\tau} | T_{k,j} > t} \leq \prob{T_{k,j} \leq t+1 | T_{k,j} > t}.
\end{equation}
But $\tau$ is a stopping time for the Markov process $( \Ei[l], l = 1,\ldots, j )_{t \geq 0}$ and the event $E_{\tau}$ depends only on the poisson processes of the Harris system for times between $\tau$ and $\tau + 1$, so, conditionally at $\left\{ \tau < \infty \right\}$, $E_{\tau}$ has the same law as $E_0$ by the strong Markov property. Hence the left hand side of \eqref{eq4} is equal to $q(\lambda)$. Finally, we have
\begin{equation*}
\prob{T_{k,j} > t+1} \leq \( 1-q(\lambda) \)\prob{T_{k,j} > t}.
\end{equation*}
The last inequality implies that there exist some deterministic positive constants $a_1 ,a_2$, depending only on $\lambda$, such that almost surely and for all $t \geq 0$ we have
\begin{equation*}
\prob{T_{k,j} > t} \leq a_1e^{-a_2t}.
\end{equation*}
The result follows with $C_{\lambda} \eqdef \int_0^{\infty} a_1e^{-a_2t}dt$.
\end{proof}

Finally, the following theorem gives the estimates that we need:

\begin{theo} \label{DensiteCouches}
For all $1 \leq j \leq i$, $\frac{\Ti[j]}{t}$ converges almost surely to a deterministic value if the process starts under $\Mi[k]$ for all $k \geq i$. Furthermore, for all $\epsilon < \frac{1}{2C_{\lambda}}$, where $C_\lambda$ is as in Lemma \ref{lem1}, we have
\begin{equation*}
\limsup_{t \to \infty}\frac{\Ni[j]}{t} \leq c_{j-1}\epsilon^{j-1}, \quad \lim_{t \to \infty}\frac{\Ti[j]}{t} \leq c_j\epsilon^{j-1},
\end{equation*}
for $1 \leq j \leq i-1$, and
\begin{equation*}
\limsup_{t \to \infty}\frac{\Ni}{t} \leq 2c_{i-1}\epsilon^{i-1}, \quad \lim_{t \to \infty}\frac{\Ti}{t} \leq 2c_i\epsilon^{i-1},
\end{equation*}
where $(c_j)_{j=1,\ldots,i}$ are constants (depending only on $\lambda$) such that $c_0 \eqdef \lambda \( 1-\lambda \)$ and $c_{j} \eqdef C_{\lambda}^{j-1}c_0$.
\end{theo}

\begin{proof}
We have seen that every $\Mi[k]$ is stationary and ergodic, so by the ergodic theorem, we have almost surely
\begin{equation} \label{eq2}
\frac{\Ti[j]}{t} \underset{t \to \infty}{\longrightarrow} \Mi[k] \left\{ \eta \in \mathcal{X}_k : \eta(2) = j \right\}
\end{equation}
and
\begin{equation} \label{eq3}
\frac{\tilde{T}_t}{t} \underset{t \to \infty}{\longrightarrow} \Mi[k] \left\{ \eta \in \mathcal{X}_k : \eta(1) \geq 2, \eta(2) = 1 \right\}.
\end{equation}
Since the distribution of the first class particles is $\nu^\lambda$ under every $\Mi[k]$, the right hand side of \eqref{eq2} is $\lambda$ if $j = 1$ and the right hand side of \eqref{eq3} is $\lambda \( 1-\lambda \)$. Using Proposition \ref{LLG}, $\Ni[1]/t$ converges to $\lambda(1-\lambda)$ almost surely.

\bigskip

\noindent
Let
\[
M^{(2)}_t \eqdef \sharp \left\{ s \in \Ncal^b_2 \cap \left[ 0, t \right] : \eta^{(1)}_s(2) \( 1-\eta^{(1)}_s(1) \) = 1 \right\}.
\]
Then almost surely $\Ni[2] \leq M^{(2)}_t$ and applying Proposition \ref{LLG}:
\begin{equation} \label{eq5}
\limsup_{t \to \infty} \frac{\Ni[2]}{t} \leq \epsilon\lambda \( 1-\lambda \) = \lim_{t \to \infty} \frac{M_t^{(2)}}{t}.
\end{equation}

Now, we need to find an upper bound for $\lim_{t \to \infty} \frac{\Ti[2]}{t}$. First, we can remark that $\Ti[2]$ can be decomposed into two parts: the time spent by initial second class particles $T^{(2)}_{t,1}$ plus the time spent by the new second class particles $T^{(2)}_{t,2}$ in site $2$. But, since $T^{(2)}_{t,1}$ is bounded by a random variable almost surely finite, it is sufficient to study $\lim_{t \to \infty} \frac{T^{(2)}_{t,2}}{t}$. As we have seen previously, $\Mi[k] \prec \nu^{\lambda + \epsilon}$. The idea is that since we know the number of second class particles created up to time $t$, it is sufficient to bound the time spent in site $2$ by one of them in the environment $\nu^{\lambda + \epsilon}$ where it is slower. But there are some difficulties. For example, at the moment where a second class particle is created, the environment in $\left\{ 2,3,\ldots \right\}$ is not dominated anymore by a Bernoulli product with density $\lambda + \epsilon$ because we know that a first class particle has to be in site $2$. To avoid this problem, we will use the following fact: if a particle of a class different than $1$ is at the site $2$ at time $t$ then it has a positive probability (depending only on $\lambda$) to be out of the system at time $t+1$. This implies Lemma \ref{lem1} which says:
\begin{equation} \label{eq6}
\esp{T_l} \leq C_{\lambda},
\end{equation}
where $T_l$ is the total time spent by the $l$-th second class particle at site $2$ and $C_{\lambda}$ is a constant. Take any $\beta > \epsilon \lambda \( 1-\lambda \)$ and $\tau \eqdef \infset{\forall s \geq t, N^{(2)}_s \leq \beta s}$. We have that $\tau$ is almost surely finite by \eqref{eq5} and
\begin{equation} \label{TempsNombre1}
\frac{T^{(2)}_{t,2}}{t} \indic{\{\tau \leq t\}} \leq \frac{1}{t} \sum_{i=1}^{\Ni[2]} T_i \indic{\{\tau \leq t\}} \leq \frac{1}{t}\sum_{i=1}^{\beta t} T_i \indic{\{\tau \leq t\}}.
\end{equation}
Taking expectation in both sides, it leads to
\begin{equation} \label{TempsNombre2}
\esp{\frac{T^{(2)}_{t,2}}{t} \indic{\{\tau \leq t\}}} \leq \frac{1}{t} \sum^{\beta t}_{i=1} \esp{T_i \indic{\{\tau \leq t\}}} \underset{\eqref{eq6}}{\leq} \beta C_{\lambda}.
\end{equation}
Hence, by dominated convergence we have almost surely
\begin{equation} \label{TempsNombre3}
\lim_{t \to \infty} \frac{\Ti[2]}{t} = \lim_{t \to \infty} \frac{T^{(2)}_{t,2}}{t} = \lim_{t \to \infty} \esp{\frac{T^{(2)}_{t,2}}{t} \indic{\{\tau \leq t\}}} \leq \beta C_{\lambda}.
\end{equation}
The above inequality is true for all $\beta >  \epsilon \lambda \( 1-\lambda \)$, thus we also have
\begin{equation*}
\lim_{t \to \infty} \frac{\Ti[2]}{t} \leq \epsilon \lambda \( 1-\lambda \) C_{\lambda}.
\end{equation*}
Let now $c_2 \eqdef C_{\lambda}c_1$ and by induction, using exactly the same arguments, we have for all $1 \leq j \leq i-1$:
\begin{equation*}
\limsup_{t \to \infty} \frac{\Ni[j]}{t} \leq c_{j-1}\epsilon^{j-1},
\end{equation*}
and
\begin{equation*}
\lim_{t \to \infty} \frac{\Ti[j]}{t} \leq c_j\epsilon^{j-1},
\end{equation*}
where $c_{j} \eqdef C_{\lambda}^{j-1} \lambda \( 1-\lambda \)$.

Finally, let $\alpha \eqdef \limsup_{t \to \infty} \frac{N_{t}^{(i)}}{t}$. Doing the same computation as in \eqref{TempsNombre1}, \eqref{TempsNombre2} and \eqref{TempsNombre3}, we get:
\begin{equation*}
\lim_{t \to \infty} \frac{T_{t}^{(i)}}{t} \leq \alpha C_{\lambda}.
\end{equation*}
Consequently,
\begin{equation*}
\lim_{t \to \infty} \frac{T_{t}^{(i-1)}+T_{t}^{(i)}}{t} \leq c_{i-1}\epsilon^{i-1}+\alpha C_{\lambda},
\end{equation*}
which implies as in \eqref{eq5}:
\begin{equation*}
\limsup_{t \to \infty} \frac{N_{t}^{(i)}}{t} = \alpha \leq (c_{i-1}\epsilon^{i-1}+\alpha C_{\lambda})\epsilon.
\end{equation*}
Since $\epsilon < \frac{1}{2C_{\lambda}}$, we have $\alpha \leq 2c_{i-1}\epsilon^{i}$ and
\begin{equation*}
\lim_{t \to \infty} \frac{T_{t}^{(i)}}{t} \leq 2c_{i}\epsilon^{i}.
\end{equation*}
\end{proof}

Now, let $\bar{N}_t^{(i-1)}$ and $\bar{N}_t^{(i)}$ be the number of particles entered in the system between $0$ and $t$ for the processes $(\bar{\eta}^{(i-1)}_t)_{t \geq 0}$ and $(\bar{\eta}^{(i)}_t)_{t \geq 0}$. We deduce from the above theorem, by summing estimates, that $\limsup_{t \to \infty} \frac{\bar{N}_t^{(i)}-\bar{N}_t^{(i-1)}}{t} = O \( \epsilon^{i-1} \)$.


\subsection{The asymptotic flux at the first order}

In this section, we consider the particle system with the generator given by \eqref{gms} for $i = 3$. In the previous section we have seen that if we want to compute the limit, at the order $\epsilon$, of the total number of new particles at time $t$ divided by $t$, then it is enough to compute this limit only for the first and second class particles. In other words, if $N_t^{T,(j)}$ denotes the number of $j$-particles entered between $0$ and $t$ which are still alive at time $t$, then:
\begin{align*}
\limsup_{t \to \infty} \frac{N_t^{T,(1)}+N_t^{T,(2)}+N_t^{T,(3)}}{t} &= \limsup_{t \to \infty} \frac{N_t^{T,(1)}+N_t^{T,(2)}}{t} + o(\epsilon),\\
												&= \lambda(1-\lambda) + \limsup_{t \to \infty} \frac{N_t^{T,(2)}}{t} + o(\epsilon).
\end{align*}
We will use the notation $N_t$ rather than $N_t^{T,(2)}$ for the number of second class particles because there will be no possible confusion. The aim of this section is to prove a law of large numbers for $N_t$ and to compute the limit at the order $\epsilon$. Let $c > 0$ such that $\lambda + c < \frac{1}{2}$ and $\epsilon \in \left[ 0, c \right]$. Let us introduce some notations. Let $X_i(t)$ be the position at time $t$ of the $i$-th second class particle entered in the system, with the convention $X_i(t) \eqdef 0$ if the corresponding particle either is not yet born or is died at time $t$. We define
\begin{align*}
&T_i^e \eqdef \infset{ X_i(t) = 1 }, &&T_i^s \eqdef \supset{ X_i(t) \geq 1 },\\
&S_i(t) \eqdef \indic{ X_i(t) \geq 1 } \hspace{1.5cm} \text{ and } &&S_i \eqdef \indic{ T_i^s = \infty }.
\end{align*}

\begin{figure}[ht]
\centering
\includegraphics[width=0.6\textwidth]{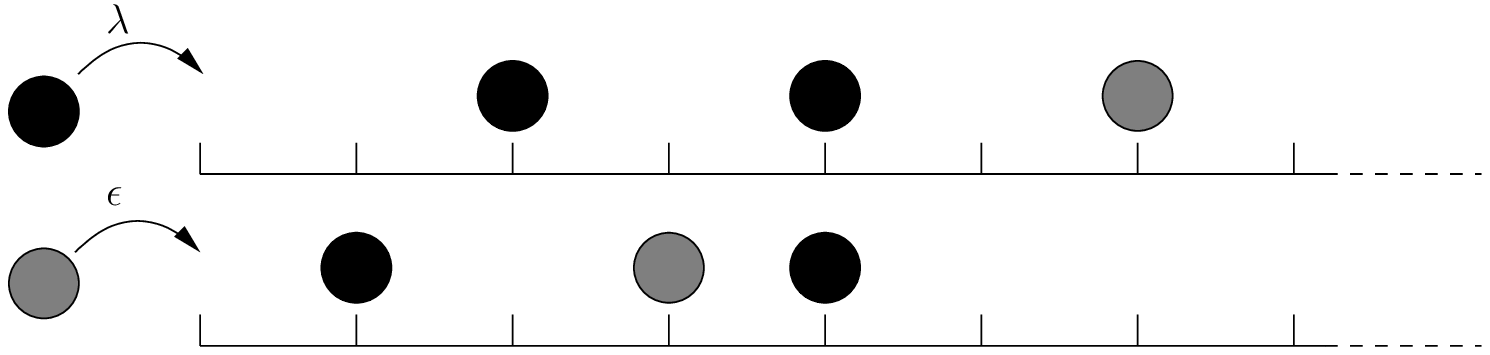}
\parbox[c]{12cm}{\caption{\textit{First class particles, in black, enter with rate $\lambda$ whatever is the configuration in $\left\{ 2,3,\ldots \right\}$ and second class particles, in grey, enter with rate $\epsilon$ if the site $2$ is occupied by a first class particle.}}
\label{fig:Tauxentree}}
\end{figure}

At time $0$ we start with the invariant measure $\Mi[2]$ and we have $\nu^\lambda \prec \Mi[2] \prec \nu^{\lambda + c}$. Since the dynamic is monotone, we can couple our process $\eta_.$ with a TASEP($\lambda$) $\eta_.^{inf}$ and a TASEP($\lambda + c$) $\eta_.^{sup}$ such that for all $t \geq 0$, we have $\eta_t^{inf} \leq \eta_t \leq \eta_t^{sup}$ almost surely. We define a new particle system with state space $\{0,1,(2,i)_{i \geq 1}\}$ and generator:
\begin{align*}
\begin{split}
\bar{\Omega}_\nu f(\eta) &\eqdef \nu\indic{\eta(1) \neq 1} \( f(\eta_{1\to1}) - f(\eta) \) + \epsilon\indic{\eta(1) \neq 1,\eta(2) = 1} \( f(\eta_{2\to1}) - f(\eta) \),\\
			        & \hspace{5mm} + \sum_{x=1}^\infty \indic{\eta(x) \neq 0,\eta(x+1) \neq 1} \( f(\eta_{x,x+1}) - f(\eta) \),
\end{split}
\end{align*}
for all cylinder function $f$, where
\[
\eta_{1\to1}(z) \eqdef \left\{ \begin{array}{cl}
					1 & \text{if } z=1,\\
					\eta(z) & \text{otherwise,}
				       \end{array} \right.
\]
and
\[
\eta_{2\to1}(z) \eqdef \left\{ \begin{array}{cl}
					(2,1) & \text{if } z=1 \text{ and } \eta(1) = 0,\\
					(2,i+1) & \text{if } z=1 \text{ and } \eta(1) = (2,i),\\
					\eta(z) & \text{otherwise.}
				       \end{array} \right.
\]
This particle system has the following description: particles of type $1$ are first class particles and they have the generator $\Omega_\nu$; particles of type $2$ are second class particles, they enter with rate $\epsilon$ if there is a first class particle at site $2$ and they do not interact with any other second class particle. To link our process to the above process, we proceed as follows: at each time for which a second class particle enters in the real system, we add a \emph{supplementary} particle to each of the processes $\eta_.^{inf}$ and $\eta_.^{sup}$ if it is possible and we denote by $X_i^{inf}(t)$ and $X_i^{sup}(t)$ the corresponding trajectories of the particles associated to $X_i$. What we mean by \emph{supplementary} particle is that these particles does not belong to the processes $\eta_.^{inf}$ and $\eta_.^{sup}$, i.e., they move as second class particles and they do not interact with each other. In particular, a given site in one of these two processes can either be empty, or contains one first class particle, or contains one or more supplementary particles. Then the process $\eta_.^{inf}$ plus supplementary particles has the generator $\bar{\Omega}_\lambda$ and the process $\eta_.^{sup}$ plus supplementary particles has the generator $\bar{\Omega}_{\lambda+\epsilon}$. We have almost surely $X_i^{sup}(t) \leq X_i(t) \leq X_i^{inf}(t)$ for all $t \geq 0$. We define analogously the quantities $N_t^{inf}$, $N_t^{sup}$, $T_i^{s,inf}$, $T_i^{s,sup}$, etc.

Consider the following initial configuration:  at time $0$ we put on $\ZZp \backslash \left\{ 1, 2 \right\}$ first class particles according to $\nu^\lambda$ (the product Bernoulli measure with density $\lambda$), and we put one first class particle in site $2$ and one second class particle in site $1$. This is exactly the distribution of the configuration $\eta_{T_i^e}^{inf}$ for all $i \geq 1$. Then first class particles enter in site $1$ with rate $\lambda$ and they have priority on the second class particle.
Two cases can occur: either the second class particle survives, or it dies. Let $p\( \lambda \)$ be the probability that the second class particle survives. $p$ is a non-increasing function, $p(0) = 1$, $p(\frac{1}{2}) = 0$ and $p\( \lambda \) > 0$ for all $\lambda < \frac{1}{2}$. Indeed, for the last point, it can be shown that if the second class particle survives, then it has a positive speed $1 - 2\lambda$ (see e.g. \cite{Saada1987}). We have by construction $\prob{S_i^{inf} = 1} = p(\lambda)$ and $\prob{S_i^{sup} = 1} = p(\lambda + c)$ for all $i \geq 1$. Consequently we have $p(\lambda) \leq \prob{S_i = 1} \leq p(\lambda + c)$.

In this section we prove the following law of large numbers:
\begin{theo} \label{DensiteOrdreEpsilon}
Almost surely, $\displaystyle{\lim_{\epsilon \downarrow 0} \frac{1}{\epsilon} \lim_{t \rightarrow \infty} \frac{N_t}{t} = \lambda \( 1-\lambda \) p(\lambda)}$.
\end{theo}
With the discussion at the beginning of the section, Theorem \ref{MainResult} follows.

The idea is the following: when $\epsilon$ is very small, second class particles do not interact before they are very far from the left boundary and if a second class particle is far enough from this boundary, then it survives with high probability. In other words, the effect on $N_t$ of interaction goes to $0$ with $\epsilon$. The first part of the proof is to find estimates for the process without interaction and to prove the theorem in this case. Next, we will show, for the ``true'' process, that if two second class particles meet, then they both survive with a probability going to $1$ as $\epsilon$ goes to $0$; this implies the theorem.


\subsubsection{The process without interaction}

Consider a family $\( \Ncal_\lambda^b \)_{0 \leq \lambda < \frac{1}{2}}$ of Poisson point processes such that the parameter of $\Ncal_\lambda^b$ is $\lambda$ and for all $0 \leq \lambda \leq \mu < \frac{1}{2}$, $\Ncal_\lambda^b \subset \Ncal_\mu^b$ and $\Ncal_\mu^b \backslash \Ncal_\lambda^b$ is independent of $\Ncal_\lambda^b$. Take also a family $\( \eta_0^\lambda \)_{0 \leq \lambda < \frac{1}{2}}$ of initial configurations such that $\eta_0^\lambda(2) \( 1-\eta_0^\lambda(1) \) = 1$ for all $\lambda \in \left[ 0,\frac{1}{2} \right[$, the distribution of $\eta_0^\lambda$ on $\left\{ 3,4,\ldots \right\}$ is $\nu^\lambda$, and for all $x \geq 3$ and all $0 \leq \lambda \leq \mu < \frac{1}{2}$, $\eta_0^\lambda(x) \leq \eta_0^\mu(x)$ almost surely. Then using the same Poisson point processes $\( \Ncal_x, x \geq 1 \)$ for the bulk dynamic we construct, as in section \ref{GraphicalConstruction}, the family of TASEP $(\eta_.^\lambda)_{0 \leq \lambda < \frac{1}{2}}$ such that $\eta_.^\lambda$ is a TASEP($\lambda$) and for all $t \geq 0$ and all $0 \leq \lambda \leq \mu < \frac{1}{2}$, $\eta_t^\lambda \leq \eta_t^\mu$ almost surely. At time $0$ we add a second class particle in site $1$ to each of these processes and we denote $X_\lambda(t)$ the position at time $t$ of the particle in the process $\eta_.^\lambda$ (with the convention $X_\lambda(t) = 0$ if the particle is died). We define
\[
S^\lambda \eqdef \indic{X_\lambda \text{ survives}},
\]
and
\[
T_x^\lambda \eqdef \infset{X_\lambda(t) = x},
\]
for all $x \geq 1$.

We start with an intuitive lemma which will be useful to propagate results from the process without interaction to the true process. 

\begin{lemm}
The function $p : \left[ 0,1 \right] \rightarrow \left[ 0,1 \right]$ is continuous.
\end{lemm}

\begin{proof}
We start by proving the right continuity of $p$. Since $p(\lambda) = 0$ for $\lambda \geq \frac{1}{2}$, it is sufficient to prove it on $\left[ 0,\frac{1}{2} \left[ \right. \right. $. Let $0 \leq \lambda < \frac{1}{2}$, $\epsilon' > 0$ and $0 < c < \lambda - \frac{1}{2}$. There exists some $x \geq 1$ such that
\[
\prob{S^{\lambda + c} = 0 | T_x^{\lambda + c} < \infty} < \epsilon'.
\]
Indeed, if $M \eqdef \max \left\{ X_{\lambda + c}(t), t \geq 0 \right\}$ then conditionally to $\left\{ S^{\lambda + c} = 0 \right\}$, $M$ is almost surely finite. Thus there exists $x \geq 1$ such that
\[
\prob{M \geq x | S^{\lambda + c} = 0} < \epsilon' \frac{p(\lambda + c)}{1-p(\lambda + c)},
\]
which implies
\[
\prob{S^{\lambda + c} = 0 | T_x^{\lambda + c} < \infty} = \prob{M \geq x | S^{\lambda + c} = 0} \frac{\prob{S^{\lambda + c} = 0}}{\prob{M \geq x}} < \epsilon'.
\]
Furthermore for all $\epsilon \in \left[ 0,c \right]$,
\begin{align}
\begin{split} \label{MortSachantLoin}
\prob{S^{\lambda + \epsilon} = 0 | T_x^{\lambda + \epsilon} < \infty} &= \frac{\prob{S^{\lambda + \epsilon} = 0}-1}{\prob{T_x^{\lambda + \epsilon} < \infty}} + 1,\\
												         &\leq \frac{\prob{S^{\lambda + c} = 0}-1}{\prob{T_x^{\lambda + c} < \infty}} + 1,\\
												         &< \epsilon'.
\end{split}
\end{align}
Now let $t_0 \geq 0$ such that
\begin{equation} \label{TempsAtteinte}
\prob{\sup_{t \in \left[ 0,t_0 \right]} X_\lambda(t) \geq x | T_x^\lambda < \infty} > 1-\epsilon'.
\end{equation}
We can find $0 < c' \leq c$ such that
\begin{equation} \label{PasInteraction}
\prob{\sum_{i=1}^x \( \eta_0^{\lambda + c'}(i)-\eta_0^\lambda(i) \) = 0, \hspace{2mm} \( \Ncal_{\lambda + c'}^b \backslash \Ncal_\lambda^b \) \cap \left[ 0,t_0 \right] = \emptyset} > 1-\epsilon'.
\end{equation}
Conditionally to the event
\[
A \eqdef \left\{ \sup_{t \in \left[ 0,t_0 \right]} X_\lambda(t) \geq x, \hspace{1mm} \sum_{i=1}^x \( \eta_0^{\lambda + c'}(i)-\eta_0^\lambda(i) \) = 0, \hspace{1mm} \( \Ncal_{\lambda + c'}^b \backslash \Ncal_\lambda^b \) \cap \left[ 0,t_0 \right] = \emptyset \right\},
\]
we have almost surely $T_x^{\lambda + c'} < \infty$. Thus using \eqref{MortSachantLoin}, \eqref{TempsAtteinte} and \eqref{PasInteraction}
\begin{align*}
\begin{split}
\prob{T_x^{\lambda + c'} < \infty | T_x^\lambda < \infty} &\geq \prob{T_x^{\lambda + c'} < \infty | A} \prob{A | T_x^\lambda < \infty} = \prob{A | T_x^\lambda < \infty},\\
										    &\geq \( 1-\epsilon' \)^2 > 1-2\epsilon'.
\end{split}
\end{align*}
Finally we get
\begin{align*}
\begin{split}
p(\lambda) - p(\lambda + c') &= \prob{S^{\lambda + c'} = 0, S^\lambda = 1},\\
					    &\leq \prob{S^{\lambda + c'} = 0 | T_x^{\lambda + c'} < \infty} + \prob{T_x^{\lambda + c'} = \infty | T_x^\lambda < \infty},\\
					    &< 3\epsilon'.
\end{split}
\end{align*}
For the left continuity we do the same reasoning.
\end{proof}

Now we prove Theorem \ref{DensiteOrdreEpsilon} in the case without interaction.

\begin{prop}
$\frac{N_t}{t}$ and $\frac{N_t^{inf}}{t}$ both have almost sure limits as $t$ goes to infinity and
\[
\lim_{\epsilon \rightarrow 0} \frac{1}{\epsilon} \lim_{t \rightarrow \infty} \frac{N_t^{inf}}{t} = \lambda \( 1-\lambda \) p(\lambda).
\]
\end{prop}

\begin{proof}
The convergences to almost sure limits are a consequence of Proposition \ref{LLG}. Indeed, for example $N_t$ counts the number of elements of $\Ncal_2^b$ for which $\eta_t(1) \neq 1$ and $\eta_t(2) = 1$ minus the number of elements of $\Ncal_1^b$ for which $\eta_t(1) = 2$.

Then we have:
\begin{equation} \label{eq10}
\frac{N_t^{inf}}{t} = \frac{1}{t} \sum_{i=1}^{N_t^e} S_i(t) \geq \frac{1}{t} \sum_{i=1}^{N_t^e} S_i.
\end{equation}
Furthermore, since $\esp{S_i} = p(\lambda)$ for all $i \geq 1$, the expectation of the right hand side of \eqref{eq10} converges to $\lambda(1-\lambda)p(\lambda)\epsilon$. But if we denote $(t_n)_{n \geq 1}$ the successive times for which the $N_{t_n}^e$-th particle is exactly the $n$-th particle which will survive, then we have:
\begin{equation*}
\frac{N_{t_n}^{inf}}{t_n} = \frac{1}{t_n} \sum_{i=1}^{N_{t_n}^e} S_i = \frac{n}{t_n}.
\end{equation*}
Thus if $n_t \eqdef \sup\{n \geq 1 : t_n \leq t\}$ then, since $\frac{t_n}{n}$ converges almost surely to $(\lim_{t\to\infty} \frac{N_t^{inf}}{t})^{-1}$, we have almost surely:
\begin{equation} \label{eq11}
\lim_{t\to\infty} \frac{n_t - N_t^{inf}}{t} = 0.
\end{equation}
On the other side, $n_t$ is exactly the number at time $t$ of particles which are already born and which will survive, i.e., $n_t = \sum_{i=1}^{N_t^e} S_i$ almost surely. This shows that the right hand side of \eqref{eq10} converges also almost surely to $\lambda(1-\lambda)p(\lambda)\epsilon$ and \eqref{eq11} implies that $\frac{N_t^{inf}}{t}$ converges to the same limit.
\end{proof}


\subsubsection{Interaction implies survival}

The following lemma states that if a second class particle goes far enough, then it survives with high probability.

\begin{lemm} \label{EstimationSurvie}

For all $\epsilon' > 0$, there exists $x_0$ (depending only on $\lambda$ and $\epsilon'$) such that if $c$ is small enough, then for all $i \geq 1$
\begin{equation*}
\prob{T_i^s < \infty; \exists t \geq 0, X_i(t) \geq x_0} < \epsilon'.
\end{equation*}
\end{lemm}

\begin{proof}
We start by proving the same result for $X^{inf}$. Let $M \eqdef \sup \left\{ X_i^{inf}(t), t \geq 0 \right\}$. Conditionally on $\left\{ T_i^{s,inf} < \infty \right\}$, $M$ is almost surely finite, thus we can choose $x_0$ such that
\[
\prob{M \geq x_0 | T_i^{s,inf} < \infty} < \frac{\epsilon'}{2}.
\]
Hence
\[
\prob{T_i^{s,inf} < \infty; \exists t \geq 0, X_i^{inf}(t) \geq x_0} = \prob{M \geq x_0 | T_i^{s,inf} < \infty}\prob{T_i^{s,inf} < \infty} < \frac{\epsilon'}{2}.
\]
Furthermore, since the law of $X_i^{inf}$ is the same for all $i \geq 1$ (because there is no interaction among supplementary particles), we can choose the same $x_0$ for all $i \geq 1$.
Then we have
\begin{align*}
\prob{T_i^s < \infty; \exists t \geq 0, X_i(t) \geq x_0} &\leq \prob{T_i^{s,inf} < \infty; \exists t \geq 0, X_i^{inf}(t) \geq x_0}\\
									       &\hspace{5mm} + \prob{S_i \neq S_i^{inf}},\\
									       &< \frac{\epsilon'}{2} + p(\lambda) - p(\lambda + c),\\
									       &< \epsilon',
\end{align*}
if $c$ is small enough.
\end{proof}

For $x \geq 1$, let $T_x \eqdef \infset{X_i(t) = x}$ (we omit the dependance on $i$ in the notation because there will be no possible confusion). We can deduce from this lemma a stronger form of the same estimate:

\begin{coro} \label{EstimationSurvieForte}
Let $x \geq 1$. For all $\epsilon' > 0$, there exists $x_1$ depending only on $\lambda$, $\epsilon'$ and $x$ such that if $c$ is small enough then for all $i \geq 1$
\begin{equation*}
\prob{T_{x_1} < \infty; \exists t \geq T_{x_1}, X_i(t) \leq x} < \epsilon'.
\end{equation*}
\end{coro}

\begin{proof}
We will use the same method as in Lemma \ref{lem1}. Let $E_t$ be the following event on the Poisson point processes of the Harris system during the time space $[t,t+1]$:
\begin{itemize}
\item one first class particle enters in site $1$ and moves to site $x$;
\item then one first class particle enters and moves to site $x-1$;
\item we continue in the same way until $x$ first class particles are entered in the system and they have moved until that the box $\{1,\ldots,x\}$ is full;
\item finally we impose that $\Ncal_x \cap [t,t+1] = \emptyset$.
\end{itemize}
Then $q_x(\lambda) \eqdef \prob{E_t}$ depends only on $\lambda$ and $x$, is positive and under this event every second class particle which was in the box $\left\{ 1, \ldots, x \right\}$ at time $t$ is died at time $t+1$.

Now let $x_1$ be given by Lemma \ref{EstimationSurvie} such that $\prob{T_i^s < \infty, \exists t \geq 0, X_i(t) \geq x_1} < \epsilon' q_x(\lambda)$ and define $T_x^+ \eqdef \inf \left\{ t \geq T_{x_1} : X_i(t) = x \right\}$. Then we have $\prob{T_i^s < \infty | T_x^+ < \infty} \geq q_x(\lambda)$. This implies
\[
\prob{T_{x_1} < \infty, \exists t \geq T_{x_1}, X_i(t) \leq x} = \prob{T_x^+ < \infty} = \frac{\prob{T_i^s < \infty, T_x^+ < \infty}}{\prob{T_i^s < \infty | T_x^+ < \infty}} < \epsilon'.
\]
\end{proof}

The next lemma states that if we fix $x \geq 1$, then the probability that two second class particles meet in the box $\left\{ 1, \ldots, x \right\}$ goes to $0$ with $\epsilon$.

\begin{lemm}
Let $T_{i+1 \rightarrow i}$ be the first time at which the $\(i+1\)$-th second class particle tries to jump on the site occupied by the $i$-th second class particle. Then for all fixed $x \geq 1$,
\begin{equation*}
\prob{T_{i+1 \rightarrow i} < \infty, X_i(T_{i+1 \rightarrow i}) \leq x} \tend[0]{\epsilon} 0, \hspace{5mm} \text{ uniformly in $i$}.
\end{equation*}
\end{lemm}

\begin{proof}
Fix $\epsilon' > 0$ and let $x_1$ and $0< c_0 < \frac{1}{2} - \lambda$ be given by Corollary \ref{EstimationSurvieForte} such that
\begin{equation} \label{DistanceDeRetour}
\prob{T_{x_1} < \infty, \exists t \geq T_{x_1}, X_i(t) \leq x} < \epsilon', \hspace{5mm} \text{ for all $\epsilon \leq c_0$}.
\end{equation}
Then $x_1$ and $c_0$ depend only on $\lambda$ and $\epsilon'$ (and $x$). We have:
\begin{align}
\begin{split} \label{TempsDansBoite}
\prob{\exists s \geq t, X_i(s) \in \left\{ 1,\ldots,x \right\}} &\leq \prob{\exists s \geq t, X_i(s) \in \left\{ 1,\ldots,x \right\}, T_{x_1} \leq t}\\
							      & \hspace{5mm} + \prob{X_i(t) \geq 1, T_{x_1} > t },\\
							      &\leq \prob{T_{x_1} < \infty, \exists s \geq T_{x_1}, X_i(s) \leq x}\\
							      & \hspace{5mm} + \prob{X_i(s) \in \left\{ 1,\ldots,x_1 \right\}, \forall s \in \left[ 0,t \right]}.
\end{split}
\end{align}
As in Lemma \ref{lem1}, we have
\[
\prob{X_i(s) \in \left\{ 1,\ldots,x_1 \right\}, \forall s \in \left[ 0,t+1 \right]} \leq \( 1-q_{x_1}(\lambda) \)\prob{X_i(s) \in \left\{ 1,\ldots,x_1 \right\}, \forall s \in \left[ 0,t \right]},
\]
which implies the existence of a constant $C > 0$ depending only on $\lambda$ and $\epsilon'$ such that
\[
\prob{X_i(s) \in \left\{ 1,\ldots,x_1 \right\}, \forall s \in \left[ 0,t \right]} \leq e^{-Ct}.
\]
Finally, using \eqref{DistanceDeRetour} and \eqref{TempsDansBoite}, there exists some deterministic $t_0 \geq 0$, depending only on $\lambda$ and $\epsilon'$, such that
\[
\prob{\exists s \geq t, X_i(s) \in \left\{ 1,\ldots,x \right\}} < 2\epsilon',
\]
for all $t \geq t_0$ and $\epsilon \leq c_0$.

Besides, if we define $T_\epsilon$ as the time elapsed between $T_i^e$ and the first jumping time of $\Ncal_2^b$ greater than $T_i^e$, then $T_\epsilon$ is an exponential random variable with parameter $\epsilon$ independent of the trajectory of $X_i$. As a consequence, we have
\begin{align*}
\prob{T_{i+1 \rightarrow i} < \infty, X_i(T_{i+1 \rightarrow i}) \leq x} &\leq \prob{\exists t \geq T_\epsilon, X_i(t) \in \left\{ 1,\ldots,x \right\}},\\
												      &\leq \prob{\exists t \geq T_\epsilon, X_i(t) \in \left\{ 1,\ldots,x \right\}, T_\epsilon > t_0}\\
												      &\hspace{5mm}+ \prob{T_\epsilon \leq t_0},\\
												      &< 2\epsilon' + 1 - e^{-\epsilon t_0}.
\end{align*}
Finally we have $\prob{T_{i+1 \rightarrow i} < \infty, X_i(T_{i+1 \rightarrow i}) \leq x} \tend[0]{\epsilon} 0$ uniformly in $i$.
\end{proof}

Now we are able to prove that when a second class particle meets another one, both survive with a probability going to $1$ as $\epsilon$ goes to $0$.

\begin{coro}
\begin{equation} \label{RencontreImpliqueSurvie}
\prob{T_{i+1 \rightarrow i} < \infty, T_{i+1}^s < \infty} \tend[0]{\epsilon} 0, \hspace{5mm} \text{ uniformly in $i$}.
\end{equation}
\end{coro}

\begin{proof}
Fix $\epsilon' > 0$ and let $x_0$ be given by lemma \ref{EstimationSurvie}. We have
\begin{align*}
\prob{T_{i+1 \rightarrow i} < \infty, T_{i+1}^s < \infty}  &= \prob{T_{i+1 \rightarrow i} < \infty, T_{i+1}^s < \infty, X_i(T_{i+1 \rightarrow i}) \leq x_0}\\
										& \hspace{5mm} + \prob{T_{i+1 \rightarrow i} < \infty, T_{i+1}^s < \infty, X_{i+1}(T_{i+1 \rightarrow i}) \geq x_0},\\
										&\leq \prob{T_{i+1 \rightarrow i} < \infty, X_i(T_{i+1 \rightarrow i}) \leq x_0}\\
										& \hspace{5mm} + \prob{T_{i+1}^s < \infty, \exists t \geq 0, X_{i+1}(t) \geq x_0},\\
										&< 2\epsilon',
\end{align*}
if $\epsilon$ is small enough.
\end{proof}


\subsubsection{The proof of Theorem \ref{DensiteOrdreEpsilon}}

Fix $\epsilon' > 0$ and use \eqref{RencontreImpliqueSurvie} to find $\epsilon > 0$ small enough to have
\[
\prob{T_{i+1 \rightarrow i} < \infty, T_{i+1}^s < \infty} < \epsilon'.
\]
We have already seen that both $\frac{N_t}{t}$ and $\frac{N_t^{inf}}{t}$ converge to almost sure limits and that $\frac{1}{\epsilon} \lim_{t \rightarrow \infty} \frac{N_t^{inf}}{t}$ converges almost surely to $\lambda \( 1-\lambda \) p(\lambda)$ as $\epsilon$ goes to $0$. We also have $\lim_{t \rightarrow \infty} \frac{N_t^e}{t} \leq \lambda \( 1-\lambda \) \epsilon$, where $N_t^e$ is the number of second class particles entered in the system up to time $t$. Thus if we define
\[
\tau \eqdef \infset{\forall s \geq t, \frac{N_s^e}{s} \leq \( \lambda \( 1-\lambda \) + 1 \) \epsilon},
\]
then $\tau$ is almost surely finite and $N_t^{inf} - N_t = \sum_{i=1}^{N_t^e} \indic{S_i^{inf}(t) = 1, S_i(t) = 0}$ which implies
\begin{align}
\begin{split} \label{MajorationDifference}
\esp{\frac{N_t^{inf} - N_t }{t} \indic{\tau \leq t}} &\leq \frac{1}{t} \sum_{i=1}^{\( \lambda \( 1-\lambda \) + 1 \) \epsilon t} \prob{S_i^{inf}(t) = 1, S_i(t) = 0, \tau \leq t},\\
								      &\leq \frac{1}{t} \sum_{i=2}^{\( \lambda \( 1-\lambda \) + 1 \) \epsilon t} \prob{T_{i \rightarrow i-1} < \infty, T_i^s < \infty},\\
								      &\leq \( \lambda \( 1-\lambda \) + 1 \) \epsilon \epsilon',
\end{split}
\end{align}
and, by dominated convergence theorem, the left hand side of \eqref{MajorationDifference} converges to $\lim_{t \rightarrow \infty} \frac{N_t^{inf}}{t} - \lim_{t \rightarrow \infty} \frac{N_t}{t}$ as $t$ goes to infinity. Hence, dividing by $\epsilon$, we get
\begin{equation*}
0 \leq \lambda \( 1-\lambda \) p(\lambda) - \frac{1}{\epsilon} \lim_{t \rightarrow \infty} \frac{N_t}{t} \leq  \( \lambda \( 1-\lambda \) + 1 \) \epsilon'.
\end{equation*}
Since $\epsilon'$ was arbitrary we can conclude.


\bibliographystyle{plain}
\bibliography{Biblio}

\end{document}